\documentclass[notitlepage,12pt]{revtex4-1}
\usepackage{amsthm}
\usepackage{lmodern}
\usepackage[english]{babel}
\usepackage{amsmath}
\usepackage{amssymb}
\usepackage{mathptmx} 
\usepackage{color}
\usepackage[babel]{microtype}
\usepackage{graphicx}
\usepackage{caption}
\usepackage{subcaption}
\usepackage{csquotes}
\usepackage{float}
\usepackage[all,cmtip]{xy}
\usepackage{bm}
\usepackage[a4paper, left=3cm, right=3cm]{geometry}
\usepackage{enumitem}
\newtheorem{theorem}{\bf Theorem}[section]
\newtheorem{lemma}[theorem]{\bf Lemma}
\newtheorem{definition}[theorem]{\bf Definition}
\newtheorem{corollary}[theorem]{\bf Corollary}
\newtheorem{proposition}[theorem]{\bf Proposition}

\begin{document}

\title{Constructing links of isolated singularities of polynomials $\mathbb{R}^4\to\mathbb{R}^2$}

\author{
Benjamin Bode}

\address{H H Wills Physics Laboratory, University of Bristol, Bristol BS8 1TL, UK\\
benjamin.bode@bristol.ac.uk
}




\begin{abstract}
We show that if a braid $B$ can be parametrised in a certain way, then previous work \cite{bode:2016polynomial} can be extended to a construction of a polynomial $f:\mathbb{R}^4\to\mathbb{R}^2$ with the closure of $B$ as the link of an isolated singularity of $f$, showing that the closure of $B$ is real algebraic. In particular, we prove that closures of squares of strictly homogeneous braids and certain lemniscate links are real algebraic. We also show that the constructed polynomials satisfy the strong Milnor condition, providing an explicit fibration of the complement of the closure of $B$ over $S^1$.
\end{abstract}
\maketitle

\section{Introduction}\label{sec:intro}
It is well known that the links that arise as links around isolated singular points of complex plane curves, also known as algebraic links, are certain iterated torus links \cite{brauner:1928geometrie}, \cite{burau:1933knoten}, \cite{burau:1934verkettungen}, \cite{kahler:1929verzweigung}, \cite{zariski:1932topology}. 
These form a special class of fibered links for which their defining property implies several interesting results. For example the monodromy of an algebraic link can be chosen to be without any fixed point \cite{acampo:1975mono},\cite{le:1975mono} and algebraic links are determined by their Alexander polynomials \cite{le:1972noeuds}. Furthermore, there appear to be deep connections between the HOMFLY-homology of an algebraic link, DAHA-superpolynomials and the Hilbert scheme or the Jacobian varitey of the corresponding complex curve \cite{cp:2016daha}, \cite{ors:2012hilbert}, \cite{os:2012hilbert}.

In this paper we consider polynomial maps $f:\mathbb{R}^4\to\mathbb{R}^2$ rather than complex plane curves. We call a point $x\in\mathbb{R}^4$ a \textit{singular} point or a singularity of $f$ if the Jacobian matrix $\nabla f(x)$ is zero. If the rank of $\nabla f$ at $x$ is less than two, we call $x$ a critical point. The Cauchy-Riemann equations imply that if one considers a complex plane curve as a polynomial map $\mathbb{R}^4\to\mathbb{R}^2$, then a point is singular if and only if it is critical. In general, the singular points are by definition a subset of the critical points.  
The links around singularities of real polynomial maps are less well studied than in the complex case and in fact they are not classified yet. We show that there are infinite families of links that arise in this way. 

Singular points of polynomial maps $f:\mathbb{R}^4\to\mathbb{R}^2$ can be isolated in two very different ways. A singular point $x$ of $f$ is called \textit{weakly isolated} if it has an $\epsilon$-neighbourhood $B_{\epsilon}$ such that $x$ is the only critical point in $B_{\epsilon}$ that is also part of the vanishing set. 
The link type of the intersection of three-spheres of small enough radius about $x$ with the vanishing set of $f$ is then, exactly as in the case of complex plane curves, independent of the radius and we call the resulting 1-manifold the link of the singularity. The links that arise in this way are called weakly algebraic links. Weak isolation is not a very strong condition and Akbulut and King showed that every link is weakly algebraic \cite{ak:1981all}.

The situation changes if we impose stricter conditions. A singular point $x$ of $f$ is called \textit{isolated} if it has an $\epsilon$-neighbourhood $B_{\epsilon}$ such that $x$ is the only critical point in $B_{\epsilon}$. Again we define the link of the singularity to be the intersection of $f^{-1}(0)$ with $S^{3}_{\epsilon}$ and call the links that arise in this way real algebraic links.

Unfortunately, the terms 'real algebraic knots' and 'algebraic knots' are used in several different contexts. Closures of compositions of rational tangles are called algebraic knots \cite{conway:1970enumeration} as well as the links of isolated singular points of complex plane curves. Additionally, in our terminology the title of Akbulut and King's paper 'All knots are algebraic' \cite{ak:1981all} translates to 'All knots are weakly algebraic'. Oleg Viro and others have developed a branch of knot theory called real algebraic knot theory where the knots of interest are (projective) algebraic varieties \cite{viro:2001writhe}. Our use of the term 'real algebraic' goes back to Perron's paper \cite{perron:1982noeud} whose title translates to 'The figure eight knot is real algebraic'. We find this fitting as it stresses the analogy to the case of complex plane curves.  

Compared to weak isolation, the stronger notion of isolation is much more restrictive. In particular, Milnor showed that every real algebraic link is fibered \cite{milnor:1968singular}. Benedetti and Shiota conjectured in \cite{benedetti1998:real} that the two sets are actually identical and that every fibered link should be real algebraic. 

This conjecture is to our knowledge still open. 
Naturally, all algebraic links are real algebraic and some examples of non-algebraic links have been shown to be real algebraic, including the figure-eight knot and the Borromean links by Perron \cite{perron:1982noeud} and Rudolph \cite{rudolph:1987isolated}. Furthermore, Looijenga showed that if $K$ is fibered, then the connected sum $K\# K$ is real algebraic \cite{looijenga:1971note}. 
The proofs by Perron, Rudolph and Looijenga illustrate a difficulty of the conjecture, namely that there does not seem to be any way of proving that a given link is real algebraic other than constructing a corresponding polynomial $f$ explicitly. 

Motivated by applications to physical systems such as optical scalar fields \cite{bd:2001knotted}, \cite{dkjop:2010isolated}, non-linear field theories \cite{sutcliffe:2007knots}, quantum-mechanics \cite{berry:2001hydrogen}, liquid crystals \cite{ma:2014knotted} and topological fluids \cite{moffatt:1969degree}, we developed a construction of polynomials $f:\mathbb{R}^4\to\mathbb{R}^2$ such that $f^{-1}(0)\cap S^{3}$ is any given link \cite{bode:2016lemniscate}, \cite{bode:2016polynomial}, \cite{bode:201652}. In fact, the constructed polynomials are semiholomorphic; they can be written as a complex polynomial in complex variables $u$, $v$ and the complex conjugate $\overline{v}$.


The semiholomorphic polynomials described in \cite{bode:2016polynomial} have the desired link as their zero level set on the unit three-sphere. However, the construction does not provide any information about three-spheres of different radii and it is straightforward to see that the constructed polynomials are typically not of the desired form with a (weakly) isolated singular point at the origin and the required link around it \cite{bode:201652}.

In this paper we introduce conditions under which we can manipulate the polynomials constructed in \cite{bode:2016polynomial} to obtain a polynomial that satisfies all of Akbulut and King's properties.
Furthermore, if additional conditions are met, the singular point at the origin is shown to be not only weakly isolated, but isolated. This explicit construction allows us to show the real algebraicity of an infinite family of links. The main result of this paper is the following theorem.

\begin{theorem}
\label{even}
Let $B$ be a braid on $s$ strands such that $B=w^2$ for some strictly homogeneous braid $w$. Then there exists a function $f:\mathbb{C}^{2}\to\mathbb{C}$ such that
\begin{itemize}
\item $f$ is a polynomial in $u$, $v$ and $\overline{v}$,
\item as a map from $\mathbb{R}^{4}$ to $\mathbb{R}^{2}$ the map $f$ has an isolated singularity at the origin,
\item $f^{-1}(0)\cap S^{3}_{\rho}$ is ambient isotopic to the closure of $B$ for all positive $\rho\leq1$,
\item $\deg_{u} f=s$. 
\end{itemize}
Hence the closure of $B$ is real algebraic. 
\end{theorem}

Hopefully the proof might be extended to more links and some of the desirable properties of algebraic links can be found for links of real singularities as well.

The structure of the paper is as follows. In Section \ref{review} we review the construction of semiholomorphic polynomials whose nodal set on the unit three-sphere is any given knot or link from \cite{bode:2016polynomial}. Section \ref{weak} discusses ways of manipulating these polynomials to obtain maps with weakly isolated singularities and Section \ref{polynomial} states when the resulting maps can be taken to be polynomials. The proof of Theorem \ref{even} can be found in Section \ref{real}. In Section \ref{lemniscatesection} we give numerical indications that some links not covered by Theorem \ref{even} can be proven to be real algebraic in a similar fashion. In Section \ref{milnor} we show that the argument of the constructed polynomial $f$ is a fibration of $S^3_{\epsilon}\backslash f^{-1}(0)$ over $S^1$.

\section{Knotted nodal sets on $S^{3}$}
\label{review}

The construction of semiholomorphic polynomials with knotted zero level sets on the unit three-sphere in \cite{bode:2016polynomial} is built around the idea of Alexander's theorem that every link is the closure of some braid \cite{alexander:1923lemma}. It is convenient to study braids (or rather their equivalence classes) as elements of the Artin Braid Group and the algebraic properties have led to important insights in knot theory. In our construction however, it is more convenient to consider a braid $B$ on $s$ strands as a collection of $s$ disjoint parametric curves
\begin{equation}
\label{eq:para}
\bigcup_{j=1}^{s}\left(X_{j}(t),Y_{j}(t),t\right),\quad t\in[0,2\pi],
\end{equation} 
such that there exists a permutation $\pi_{B}\in S_{s}$ such that 
\begin{equation}(X_{j}(2\pi),Y_{j}(2\pi))=(X_{\pi_{B}(j)}(0),Y_{\pi_{B}(j)}(0)) 
\end{equation}
for all $j\in\{1,2,\ldots,s\}$.

For a parametrisation as in Equation (\ref{eq:para}), we can define a family of polynomials $g_{a,b}:\mathbb{C}\times[0,2\pi]\to\mathbb{C}$,
\begin{equation}
\label{eq:gab}
g_{a,b}(u,t)=\prod_{j=1}^{s}\left(u-aX_{j}(t)-biY_{j}(t)\right),\qquad t\in[0,2\pi]
\end{equation}
with $a,b>0$. By construction the zero level set of $g_{a,b}$ is the braid $B$ for any choice of $a$ and $b$.

Let $\mathcal{C}$ be the set of link components of the closure of $B$ and $s_{C}$ be the number of strands forming the link component $C\in\mathcal{C}$. Suppose now that the parametrisation (\ref{eq:para}) is of the form
\begin{equation}
\label{eq:fourier}
\bigcup_{C\in\mathcal{C}}\bigcup_{j=1}^{s_{C}}\left(F_{C}\left(\frac{t+2\pi j}{s_{C}}\right),G_{C}\left(\frac{t+2\pi j}{s_{C}}\right),t\right),\qquad t\in[0,2\pi],
\end{equation}
where $F_{C}$ and $G_{C}$ are trigonometric polynomials
\begin{equation}
\label{eq:trigpoly}
F_{C}(t)=\sum_{j=-N_{C}}^{N_{C}}a_{C,j}e^{ijt},\quad G_{C}(t)=\sum_{j=-M_{C}}^{M_{C}}b_{C,j}e^{ijt},
\end{equation} 
where $a_{C,-j}=\overline{a}_{C,j}$ and $b_{C,-j}=\overline{b}_{C,j}$ for all $C$ and all $j$. We often call a parametrisation like this a Fourier parametrisation.
Then by a result from \cite{bode:2016polynomial} $g_{a,b}$ can be written as a polynomial in $u$, $e^{it}$ and $e^{-it}$. This expression  is unique once it is simplified so that it does not contain any monomials with factors of the form $e^{it}e^{-it}$. Each $g_{a,b}$ is thus the restriction of a semiholomorphic polynomial $f_{a,b}$ to the set $\mathbb{C}\times S^{1}$, i.e. there is a  polynomial $f_{a,b}:\mathbb{C}^{2}\to\mathbb{C}$ in $u$, $v$ and $\overline{v}$, such that $f_{a,b}(u,e^{it})=g_{a,b}(u,t)$ for all $u$ and $t$. Hence $f_{a,b}^{-1}(0)\cap(\mathbb{C}\times S^{1})$ is the closure of $B$ for all $a$ and $b$.

We showed in \cite{bode:2016polynomial} that if $\lambda>0$ is chosen sufficiently small, then the zero level set of $f_{\lambda a,\lambda b}$ on the unit three-sphere $S^{3}$ is the closure of $B$. Knowing how to find a braid parametrisation as in Equation (\ref{eq:fourier}) and how small $\lambda$ has to be makes this construction algorithmic.
In fact the nature of our construction allows us to establish certain properties of the constructed functions such as a certain stability of the link under small perturbations of the coefficients and a bound on the polynomial degree in terms of braid data \cite{bode:2016polynomial}.

\section{Weakly isolated singularities}
\label{weak}

In order to prove Theorem \ref{even} we need to make certain alterations to the construction described in \cite{bode:2016polynomial} and outlined in Section \ref{review}. Let $B$ be a braid on $s$ strands. Then there is an algorithm that generates a family of semiholomorphic polynomial $f_{a,b}:\mathbb{C}^{2}\to\mathbb{C}$ in $u$, $v$ and $\overline{v}$ such that $f_{a,b}^{-1}(0)\cap S^{3}$ is the closure of $B$ if $a$ and $b$ are small enough. However, it does not necessarily have a weakly isolated singularity and the zero level set on three-spheres of small radii is typically different from the closure of $B$.

In some cases however, the function $f_{a,b}$ can be manipulated in such a way that it has a weakly isolated singular point at the origin with the desired link around it.

\begin{lemma}
\label{general}
Let $q_{1},q_{2}\in\mathbb{R}_{\geq 0}$ and $g_{a,b}$ constructed as in Equation (\ref{eq:gab}) from a braid parametrisation (\ref{eq:fourier}) of the braid $B$ on $s$ strands. Let $f_{a,b}$ be constructed as in Section \ref{review} and $k\geq \deg f_{a,b}/2s$.
We define 
\begin{equation}
\label{eq:firstscaling}
P_{a,b,k}:\mathbb{C}\times[0,1]\times S^1\to\mathbb{C},\quad P_{a,b,k}(u,r,t)=r^{2sk} g_{r^{q_{1}} a,r^{q_{2}} b}\left(\frac{u}{r^{2k}},t\right).
\end{equation}
Then, since $P_{a,b,k}(u,0,t)=u^s$ for all $t\in[0,2\pi]$, we can define $p_{a,b,k}:\mathbb{C}^2\to\mathbb{C}$, $p_{a,b,k}(u,re^{it})=P_{a,b,k}(u,r,t)$. For small enough $a,b>0$ the map $p_{a,b,k}:\mathbb{R}^4\to\mathbb{R}^2$ has a weakly isolated singular point at the origin and $p_{a,b,k}^{-1}(0)\cap S^{3}_{\rho}$ is the closure of $B$ for all $\rho\in(0,1]$.
\end{lemma}

\begin{proof}
First of all, by definition $P_{a,b,k}$ and $p_{a,b,k}$ are polynomials in $u$ for every fixed $r$ and $t$. A straightforward calculation shows that $p_{a,b,k}(0,0)=0$ and $(0,0)$ is a singular point when $p_{a,b,k}$ is viewed as a map $\mathbb{R}^4\to\mathbb{R}^2$. 

Next we need to show that the singular point at the origin is weakly isolated. Note that $p_{a,b,k}(u,0)=u^s$, so the origin is the only point with $p_{a,b,k}(u,v)=0$ and $v=0$. Now let $(u,v)=(u,re^{it})\in\mathbb{C}\times \mathbb{C}\backslash\{0\}$ be in the zero level set. We find that $p_{a,b,k}(u,re^{it})=r^{2sk}g_{r^{q_{1}}a,r^{q_{2}}b}(u/r^{2k},r,t)$ and since for every $r\in(0,1]$ and $t\in[0,2\pi]$ the polynomials $g_{a,b}(\bullet,t)$ for all $a,b$ only has simple roots, we get
\begin{equation}
\label{eq:simple}
\frac{\partial p_{a,b,k}}{\partial u}(u,re^{it})=r^{2k(s-1)}\frac{\partial g_{r^{q_{1}}a,r^{q_{2}}b}}{\partial u}\left(\frac{u}{r^{2k}},r,t\right)\neq 0.
\end{equation}
It follows from the Cauchy-Riemann equations that the Jacobian of $p_{a,b,k}$, again viewed as a map $\mathbb{R}^4\to\mathbb{R}^2$, has full rank at $(u,v)$ and hence the singular point at the origin is weakly isolated.

Since $p_{a,b,k}(u,1,t)=g_{a,b}(u,t)$ the same arguments as in the construction of $f_{a,b}$ apply to $p_{a,b,k}$ (cf. \cite{bode:2016polynomial}), meaning that as long as $a$ and $b$ are chosen small enough the zero level set of $p_{a,b,k}$ on the unit three-sphere is the closure of $B$. Furthermore, if $a$ and $b$ are small enough, the intersection of the zero level set of $p_{a,b,k}$ and the unit three-sphere is transverse.

In order to guarantee that the zero level set of $p_{a,b,k}$ on the three-sphere of any radius is the closure of $B$ and not just on the unit three-sphere, it suffices to show that for all $\rho\in(0,1]$ and all $x\in p_{a,b,k}^{-1}(0)\cap S^{3}_{\rho}$ the Jacobian on the three-sphere $\nabla_{S^{3}_{\rho}}p_{a,b,k}(x)$ has full rank. Since for every $v\in\mathbb{C}\backslash\{0\}$ the roots of the complex polynomial $p_{a,b,k}(\bullet,v)$ are simple, it is sufficient to show that for all $\rho\in(0,1]$ the intersection of $S^{3}_{\rho}$ with $p^{-1}_{a,b,k}(0)$ is transverse. Note that for every $t\in[0,2\pi]$ the roots of $p_{a,b,k}(\bullet,re^{it})$ are given by $r^{2k+q_{1}}a \mathrm{Re}(u_{j}(t))+ir^{2k+q_{2}}b \mathrm{Im}(u_{j}(t))$, $j=1,2,\ldots,s$, where $u_{j}(t)$ is the $j$th root of $g_{1,1}(\bullet,t)$.

Let $(u,v)$ be a point on $p_{a,b,k}^{-1}(0)$ that is not the origin, then $v=re^{it}\neq 0$. If $u=0$, then there is a $j\in\{1,2,\ldots,s\}$ such that $u_{j}(t)=0$ and then $r^{2k+q_{1}}a\mathrm{Re}(u_{j}(t))+ir^{2k+q_{2}}b \mathrm{Im}(u_{j}(t))=0$ for all $r$. Thus in the basis $(\mathrm{Re}(u),\mathrm{Im}(u),r,t)$ the vector $(0,0,1,0)$ is tangent to $p_{a,b,k}^{-1}(0)$ at $(u,v)$ and the intersection of $p_{a,b,k}^{-1}(0)$ with $S^3_{|(u,v)|}$ is transverse at $(u,v)$.

If $u\neq0$, then one tangent vector of $p_{a,b,k}^{-1}(0)$ at $(u,v)$ in the basis $(|u|,\arg(u),r,t)$ is 
\begin{equation}
\label{eq:tangent}
(2kr^{2k-1}|U|+r^{2k}(2q_{1}r^{2q_{1}-1}a^2 \mathrm{Re}(u_{j}(t))^2+2q_{2}r^{2q_{2}-1}b^2 \mathrm{Im}(u_{j}(t))^2)/|U|,0,1,0),
\end{equation}
where $|U|=\sqrt{r^{2q_{1}}a^2 \mathrm{Re}(u_{j}(t))^2+r^{2q_{2}}b^2 \mathrm{Im}(u_{j}(t))^2}$.
The vector in Equation (\ref{eq:tangent}) is for small enough $a$ and $b$ not in the tangent space of $S^{3}_{|(u,v)|}$ (which is for example spanned by $(0,1,0,0)$, $(0,0,0,1)$ and $(-r/|u|,0,1,0)$). Hence the intersection of $p_{a,b,k}^{-1}(0)$ and $S^3_{|(u,v)|}$ is transverse at $(u,v)$.

Therefore all intersections of $p_{a,b,k}^{-1}(0)$ with $S^3_{\rho}$ are transverse for all $\rho\in(0,1]$.

We have thus shown that as the radius $\rho$ of the three-sphere $S^{3}_{\rho}$ varies between zero and one the link type of $p_{\lambda a, \lambda b, k}^{-1}(0)\cap S^{3}_{\rho}$ does not change if $\lambda$ is small enough and $k$ is large enough. Hence for sufficient choices of $a$, $b$ and $k$ the zero level set of $p_{a,b,k}$ on the three-sphere of any radius at most one is the closure of $B$, which finishes the proof.
\end{proof}

Lemma \ref{general} is in general not constructing functions of the form discussed by Akbulut and King. The newly defined function $p_{a,b,k}$ is a polynomial in $u$, but it might not be possible to write it as a polynomial in $v$ and $\overline{v}$. 

If we set $q_{1}=q_{2}=0$, the resulting function
\begin{equation}
\label{eq:p}
p_{a,b,k}(u,r,t)=r^{2sk}g_{a,b}\left(\frac{u}{r^{2k}},t\right)=r^{2sk}f_{a,b}\left(\frac{u}{r^{2k}},e^{it}\right)
\end{equation}
only depends on $r$ by the scaling in the $u$-coordinate and the overall factor $r^{2sk}$. Hence the next lemma follows from Lemma \ref{general} with $q_{1}=q_{2}=0$.

\begin{lemma}
\label{weakest}
For large enough $k$ and small enough $a,b>0$ 
\begin{equation}
\label{eq:tildep}
p_{a,b,k}(u,v):=(v\overline{v})^{sk}f_{a,b}\left(\frac{u}{(v\overline{v})^k},\frac{v}{\sqrt{v\overline{v}}}\right),
\end{equation}
has a weakly isolated singular point at the origin and $p_{a,b,k}^{-1}(0)\cap S^{3}_{\rho}$ is the closure of $B$ for all $\rho\in(0,1]$.
\end{lemma}


As an example we consider the braid parametrised by
\begin{equation}
\label{eq:423}
\left(a\cos\left(\frac{2t+2\pi j}{4}\right),b\sin\left(\frac{3(2t+2\pi j)}{s}\right),t\right)\quad\quad t\in [0,2\pi],\ j=1,2,3,4
\end{equation}
with $a,b>0$. Its closure is the 2-component link $L_{6a1}$.
Defining $g_{a,b}$ as in Equation (\ref{eq:gab}) and expanding the product yields
\begin{align}
\label{423exp}
g_{a,b}(u,t)=&u^4+u^2(b^2-a^2-2iab\sin(2t))+\frac{1}{8}(a^4-2a^2b^2+b^4 \nonumber\\
&-a^2(a^2+6b^2)\cos(2t)+2a^2b^2\cos(4t)-b^4\cos(6t)+4ia^3b\sin(2t)\nonumber\\
&-4iab^3(\sin(2t)+\sin(4t))).
\end{align}
Then using de Moivre's identities $\sin(nt)=1/(2i)(e^{int}-e^{-int})$ and $\cos(nt)=1/2 (e^{int}+e^{-int})$ 
we get with $k=1$,
\begin{align}
\label{423subs}
P_{a,b,1}(u,r,t)&=r^{8}g_{r^{q_{1}}a,r^{q_{2}}b}\left(\frac{u}{r^{2}},t\right)\nonumber\\
&=u^4+u^2r^4(r^{2q_{2}}b^2-r^{2q_{1}}a^2-ab(e^{2it}-e^{-2it}))+\frac{r^8}{16}(2a^4r^{4q_{1}}-4a^2b^2r^{2q_{1}+2q_{2}}\nonumber\\
&+2b^4r^{4q_{2}}-a^2r^{2q_{1}}(a^2r^{2q_{1}}+6b^2r^{2q_{2}})(e^{2it}+e^{-2it})\nonumber\\
&+2a^2b^2r^{2q_{1}+2q_{2}}(e^{4it}+e^{-4it})-b^4r^{4q_{2}}(e^{6it}+e^{-6it})+4a^3br^{3q_{1}+q_{2}}(e^{2it}-e^{-2it})\nonumber\\
&-4ab^3r^{q_{1}+3q_{2}}(e^{2it}-e^{-2it}+e^{4it}-e^{-4it})).
\end{align}

We set $r=\sqrt{v\overline{v}}$, $e^{it}=v/\sqrt{v\overline{v}}$, $e^{-it}=v/\sqrt{v\overline{v}}$ and $q_{1}=q_{2}=0$ and obtain
\begin{align}
\label{eq:423subss}
p_{a,b,1}(u,v)&=u^4+u^2((v\overline{v})^{2}(b^2-a^2)-abv\overline{v}(v^2-\overline{v}^2))\nonumber\\
&+\frac{(v\overline{v})^4}{16}(2a^4-4a^2b^2+2b^4)+\frac{1}{16}((-a^4-6a^2b^2)(v\overline{v})^3(v^2+\overline{v}^2)\nonumber\\
&+2a^2b^2(v\overline{v})^2(v^4+\overline{v}^4)-b^4v\overline{v}(v^6+\overline{v}^6)+4a^3b(v\overline{v})^3
(v^2-\overline{v}^2)\nonumber\\
&-4ab^3((v\overline{v})^3(v^2-\overline{v}^2)+(v\overline{v})^2(v^4-\overline{v}^4))).
\end{align}
This function is easily checked to have a weakly isolated singularity at the origin. By the previous lemmas the link of the singularity is the closure of the braid parametrised by Equation (\ref{eq:423}) if $a$ and $b$ are small enough.

Note that in this example we obtain a polynomial in $u$, $v$ and $\overline{v}$. The next section discusses for which braid parametrisations this happens.

\section{Braid parametrisations leading to polynomial maps} 
\label{polynomial}
The construction described in Section \ref{weak} clearly works for any given braid and the resulting function $p_{a,b,k}$ is a polynomial in $u$, $v$, $\overline{v}$ and $\sqrt{v\overline{v}}$. It is in general not a polynomial map $\mathbb{R}^{4}\to\mathbb{R}^{2}$. In the following we investigate braid parametrisations that guarantee that $p_{a,b,k}$ is a polynomial in $u$, $v$ and $\overline{v}$ and thus of the form discussed by Akbulut and King \cite{ak:1981all}.

Let $B=w^2$ be the square of some braid word $w\in B_{s}$. Then we can find a finite Fourier parametrisation of $w$ and define the corresponding braid polynomial $g_{a,b}(u,t)$ that has $w$ as its zero level set. Making a simple change of variable results in the function $\tilde{g}_{a,b}(u,t)=g_{a,b}(u,2t)$ whose zero level set is $B=w^2$. Moreover, since $g_{a,b}$ is a polynomial in $u$, $e^{it}$ and $e^{-it}$, $\tilde{g}_{a,b}$ is a polynomial in $u$, $e^{i2t}$ and $e^{-i2t}$. Let $\tilde{f}_{a,b}$ be the semiholomorphic polynomial that results from $\tilde{g}_{a,b}$ by replacing every instance of $e^{it}$ by $v$ and every instance of $e^{-it}$ by $\overline{v}$. Then all exponents of $v$ and $\overline{v}$ in $\tilde{f}_{a,b}$ are even. We define $p_{a,b,k}$ as in Lemma \ref{weakest} using $\tilde{f}_{a,b}$. It is by construction a polynomial in $u$, $v$, $\overline{v}$ and $\sqrt{v\overline{v}}$. Since all exponents of $v$ and $\overline{v}$ of $\tilde{f}_{a,b}$ are even, so are all exponents of $\sqrt{v\overline{v}}$ and hence $p_{a,b,k}$ is a polynomial in $u$, $v$ and $\overline{v}$. In combination with Lemma \ref{weakest} this proves the following lemma.

\begin{lemma}
\label{evenweak}
Let $B=w^2$ be the square of some braid word $w\in B_{s}$. Then there exists a function $F:\mathbb{R}^{4}\to\mathbb{R}^{2}$ such that 
\begin{itemize}
\item $F$ is a polynomial in $u$, $v$ and $\overline{v}$,
\item $F$ has a weakly isolated singular point at the origin,
\item $F^{-1}(0)\cap S^{3}_{\rho}$ is the closure of $B$ for all $\rho\in(0,1]$,
\item $\deg_{u}F=s$.
\end{itemize}
\end{lemma} 
\begin{proof}
Take $F$ to be $p_{a,b,k}$ for small enough $a,b$ and large enough $k$. Then the first three properties are shown above or follow directly from Lemma \ref{weakest}. The fourth property follows directly from the construction and $\deg_{u}F=\deg_{u}\tilde{f}_{a,b}=\deg_{u}\tilde{g}_{a,b}=\deg_{u}g_{a,b}=s$.
\end{proof}



More generally, $p_{a,b,k}$ is a polynomial in $u$, $v$ and $\overline{v}$ if all exponents of $v$ and $\overline{v}$ in the corresponding polynomial $f_{a,b}$ are even or equivalently if the corresponding braid polynomial $g_{a,b}$ is a polynomial in $u$, $e^{2it}$ and $e^{-2it}$. In order to construct polynomial maps $\mathbb{R}^{4}\to\mathbb{R}^{2}$ with a weakly isolated singular point it is hence sufficient to find braid parametrisations as in Equation (\ref{eq:fourier}) that lead to a function $g_{a,b}$ of this form.

Lemma \ref{evenweak} covers an obvious case of braids that lead to a polynomial $f_{a,b}$ where all exponents of $v$ and $\overline{v}$ are even, so that $p_{a,b,k}$ is a polynomial in $u$, $v$ and $\overline{v}$. It is not the only way to achieve this though.

\begin{lemma}
\label{odd}
Let $B$ be a braid such that every link component $C$ of the closure of $B$ consists of the same number of strands $s_{C}$. Furthermore, let $2^m$ be the highest power of two dividing $s_{C}$. Let $B$ be parametrised as in Equations (\ref{eq:fourier}) and (\ref{eq:trigpoly}) satisfying
\begin{enumerate}[label=(\roman*)]
\item all $j\in\{0,1,\ldots,N_{C}\}$ with non-vanishing $a_{C,j}$ for some link component $C\in\mathcal{C}$  lie in the same residue class mod $2^{m+1}$, say $x\ \mathrm{mod}\ 2^{m+1}$
\label{con1}
\item all $j\in\{0,1,\ldots,M_{C}\}$ with non-vanishing $b_{C,j}$ for some link component $C\in\mathcal{C}$  lie in the same residue class mod $2^{m+1}$, say $y\ \mathrm{mod}\  2^{m+1}$.
\label{con2}
\end{enumerate}

Then $p_{a,b,k}$ with $q_{1}=x/2^m$ and $q_{2}=y/2^m$ as in Lemma \ref{weakest} is a polynomial in $u$, $v$ and $\overline{v}$.
\end{lemma}

\begin{proof}

We can write the polynomial $g_{a,b}$ as 
\begin{equation}
\label{eq:gpoly}
g_{a,b}(u,t)=\sum_{i=0}^{s}u^{i}\sum_{j+n=s-i}a^{j}b^{n}\sum'c_{i,j,n, \mathrm{pair}}e^{it\sum''(j'+n')/s_{C}},
\end{equation}
where $\sum'$ is the sum over pairs of tuples, one $j$-tuple of $j'\in\{0,1,\ldots,N_{C}\}$ with non-vanishing $a_{C,j'}$ for some $C\in\mathcal{C}$ and one $n$-tuple of $n'\in\{0,1,\ldots,M_{C}\}$ with non-vanishing $b_{C,n'}$ for some $C\in\mathcal{C}$ and for each such pair of such tuples $\sum''$ is taken to be the sum over all entries $j'$ of the $j$-tuple plus the sum over all entries $n'$ of the $n$-tuple. The coefficients $c_{i,j,n,\mathrm{pair}}$ are complex numbers depending on the values of $i$, $j$, $n$ and the pair of tuples.

The conditions \ref{con1} and \ref{con2} in Lemma \ref{odd} imply that Equation (\ref{eq:gpoly}) results in
\begin{equation}
g_{a,b}(u,t)=\sum_{i=0}^{s}u^{i}\sum_{j+n=s-i}a^{j}b^{n}\sum'c_{i,jn, \mathrm{pair}}e^{it(jx+ny+2^{r+1}m_{\mathrm{pair}})/s_{C}},
\end{equation}
where $m_{\mathrm{pair}}$ is an integer depending on the pair of tuples.

With this equality and $q_{1}=x/2^m$, $q_{2}=y/2^m$ the map $p_{a,b,k}$ becomes 
\begin{align}
\label{eq:ftilde}
p_{a,b,k}(u,v)&=(v\overline{v})^{sk}g_{r^{x/2^m}a,r^{y/2^m}b}\left(\frac{u}{(v\overline{v})^k},t\right)\nonumber \\ &=(v\overline{v})^{sk}\sum_{i=0}^{s}\left(\frac{u}{(v\overline{v})^{k}}\right)^{i}\sum_{j+n=s-i}\sqrt{v\overline{v}}^{\frac{xj+ny}{2^{m}}}a^j b^n \sum'c_{i,j,n,\mathrm{pair}}\left(\frac{v}{\sqrt{v\overline{v}}}\right)^{\frac{jx+ny+2^{m+1}m_{\mathrm{pair}}}{s_{C}}}.
\end{align}

What we need to show is that for every $i$, $j$, $n$ and $m_{\mathrm{pair}}$ that can appear in this expression with a non-zero coefficient the exponent of $\sqrt{v\overline{v}}$ is even and nonnegative, so that one obtains a polynomial in $u$, $v$ and $\overline{v}$. We know that $g_{a,b}$ is a polynomial in $u$, $e^{it}$ and $e^{-it}$, so for all terms that have a non-vanishing coefficient $c_{i,j,n,\mathrm{pair}}$, the exponent of $e^{it}$ given by $(jx+ny+2^{m+1}m_{\mathrm{pair}})/s_{C}$ is an integer.
This means that $(jx+ny+2^{m+1}m_{\mathrm{pair}})$ is a multiple of $s_{C}$ and hence divisible by $2^{m}$, which is by definition a divisor of $s_{C}$. It follows that $jx+ny$ is a multiple of $2^{m}$.

Now consider a monomial of $p_{a,b,k}$. The exponent of $\sqrt{v\overline{v}}$ in a monomial is given by $(jx+ny)/2^m-(jx+ny+2^{m+1}m_{\mathrm{pair}})/s_{C}$. By the remark above this is an integer and hence $(jx+ny)s_{C}/2^m-(jx+ny+2^{m+1}m_{\mathrm{pair}})$ is a multiple of $s_{C}$. We find that
\begin{equation}
 (jx+ny)s_{C}/2^m-(jx+ny+2^{m+1}m_{\mathrm{pair}})=2^{m+1}m_{\mathrm{pair}}-(jx+ny)(s_{C}/2^m-1)\equiv 0\ \mathrm{mod}\  2^{m+1},
\end{equation}
since $jx+ny$ is a multiple of $2^m$ and $2^m$ is the highest power of two dividing $s_{C}$. Thus $s_{C}/2^m$ is odd and hence $(s_{C}/2^m-1)$ is even and $(jx+ny)(s_{C}/2^m-1)$ is divisible by $2^{m+1}$.

It follows that $(jx+ny)s_{C}/2^m-(jx+ny+2^{m+1}m_{\mathrm{pair}})$ is a multiple of $2s_{C}$ and hence the exponent of $\sqrt{v\overline{v}}$, which is $(jx+ny)/2^m-(jx+ny+2^{m+1}m_{\mathrm{pair}})/s_{C}$ is even.
\end{proof}

Note that by Lemma \ref{general} $p_{a,b,k}$ has a weakly isolated singularity at the origin with the closure of $B$ as the link of the singularity. Hence Lemma \ref{general} and Lemma \ref{odd} provide a way of constructing the maps that were shown to exist by Akbulut and King for closures of braids that allow certain parametrisations.

It follows from Lemma \ref{general} and Lemma \ref{odd} that all lemniscate links, which are defined and discussed in more detail in \cite{bode:2016lemniscate} and in Section \ref{lemniscatesection}, can be constructed as links of weakly isolated singularities of semiholomorphic polynomials $p_{a,b,k}:\mathbb{R}^4\to\mathbb{R}^2$.

We are not aware of a general procedure that decides if a given link is the closure of a braid that admits a parametrisation as in Lemma \ref{evenweak} or as in Lemma \ref{odd}. An obvious obstruction is that the the links from Lemma \ref{evenweak} are 2-periodic.

The braids of the form in Lemma \ref{odd} must in particular satisfy that all Fourier frequencies with non-zero coefficient in the parametrisation of the $x$-coordinate are in the same residue class $\ \mathrm{mod}\  2$ and all frequencies with non-zero coefficient in the parametrisation of the $y$-coordinate are in the same residue class $\ \mathrm{mod}\  2$. This makes each $F_{C}$ and $G_{C}$ into an even or an odd function. These symmetries in the braid parametrisation are reflected in symmetries of the corresponding braid words.

If both $x$ and $y$ are even, then it is a parametrisation as in Lemma \ref{evenweak} and the closure must be 2-periodic. 

If the trigonometric polynomials $F_{C}$ parametrising the $x$-coordinate have only odd frequencies and all $G_{C}$ only even ones (or equivalently the other way around), then the link is a closure of a braid of the form 
\begin{equation}
B=w\Delta_{s}w\Delta^{-1}_{s},
\end{equation} 
where 
\begin{equation}\Delta_{s}=(\sigma_{s-1}\sigma_{s-2}\ldots \sigma_{1})(\sigma_{s-1}\sigma_{s-2}\ldots\sigma_{2})\ldots(\sigma_{s-1}\sigma_{s-2})\sigma_{s-1}
\end{equation} 
is the Garside element and $w$ is some braid word on $s$ strands. This condition is equivalent to requiring that the second half of the braid word $B$ is exactly the same as the first half, but with all the indices mirrored, so \begin{equation}B=\sigma_{j_{1}}^{\epsilon_{1}}\sigma_{j_{2}}^{\epsilon_{2}}\ldots\sigma_{j_{l}}^{\epsilon_{l}}\ \ \sigma_{s-j_{1}}^{\epsilon_{1}}\sigma_{s-j_{2}}^{\epsilon_{2}}\ldots\sigma_{s-j_{l}}^{\epsilon_{l}}.
\end{equation}

If both $x$ and $y$ are odd and $s_{C}$, then the link is the closure of a braid $B$ of the form $B=w\overline{w}$ for some braid word $w$, where $\overline{w}$ is the word $w$ with all signs switched, i.e. $\overline{\sigma_{j_{1}}^{\epsilon_{1}}\ldots\sigma_{j_{l}}^{\epsilon_{l}}}=\sigma_{j_{1}}^{-\epsilon_{1}}\ldots\sigma_{j_{l}}^{-\epsilon_{l}}$.
It follows that every closure of such a braid must be strongly positive amphicheiral.

If $x$ and $y$ are odd, but $s_{C}$ is even, the symmetry in the braid word is more complicated and so far we have not found a precise description. The difficulty arises because the symmetry of the parametrisation implies that a crossing at a given value of $t$ of $x$ leads to another crossing at $t$ and $-x$ (rather than $t+\pi$ as in the previous case). There is now a multitude of different cases, depending on for which values of $t$ there are strands that cross at $x=0$ (and thus do not necessarily induce an extra crossing) and how many of these strands there are. Note that several crossings could occur simultaneously at a given value of $t$ at $x=0$. This makes a description of the symmetry in terms of braid words harder than in the previous cases.

Although we are not aware of any concrete examples, we expect that not every braid with a braid word of one of the forms above satisfies the conditions in Lemma \ref{odd}. Again we are not aware of an algorithm that determines whether a given link satisfies one of these properties.  

Lemmas \ref{evenweak} and \ref{odd} allow us to construct polynomial maps $\mathbb{R}^{4}\to\mathbb{R}^{2}$ whose vanishing set on the three-sphere of any radius is a certain link.
As mentioned in the introduction the initial motivation for us to construct knotted vanishing sets of polynomials came from physics. A polynomial whose nodal set on a three-sphere is a given link can be used as an initial configuration in a variety of physical systems. 
We would like to point out that the polynomials constructed in this section can be used in this way when they are restricted to a three-sphere of radius $\rho$ . 

The constructed polynomials have a weakly isolated singular point at the origin. In both Lemmas \ref{evenweak} and \ref{odd} the resulting polynomials are semiholomorphic and its degree with respect to the complex variable $u$ is equal to the number of strands $s$ used in the construction.

In the following we investigate if some of the constructed polynomials have in fact an isolated singular point rather than only a weakly isolated one.

\section{Real algebraic links}
\label{real}

Sections \ref{weak} and Section \ref{polynomial} describe an explicit construction of real polynomial maps with weakly isolated singularities, namely $p_{a,b,k}$. It is a natural question if in some cases the functions constructed in this manner have in fact an isolated singularity rather than only a weakly isolated singularity. A link $L$ for which this is the case is then by definition real algebraic.

\subsection{The proof of Theorem \ref{even}}

Since the real algebraic links are a (conjecturally not proper) subset of the fibered links, it is clear that in general the singular point of $p_{a,b,k}$ is not isolated. One natural family of links to consider here consists of the links which are closures of braids that come with a natural fibration. Consider a braid parametrised as in Equation (\ref{eq:para}) and such that $g_{a,b}$ does not have any argument-critical points, i.e. for all $t\in[0,2\pi]$ and $u\in\mathbb{C}$ we have $\nabla_{\mathbb{R}^2\times[0,2\pi]}\arg(g_{a,b})\neq (0,0,0)$. Then $\arg(g_{a,b})$ extends to a fibration map from $S^{3}\backslash L$ to $S^{1}$, where $L$ is the closure of $B$. Hence $L$ is fibered. 
It turns out that a parametrisation of this form is almost sufficient to allow a construction of $L$ as the link of an isolated real singularity.



\begin{definition}A braid on $s$ strands is called \textit{strictly homogeneous} if it can be written as a word $\sigma_{i_{1}}^{\varepsilon_{1}}\sigma_{i_{2}}^{\varepsilon_{2}}\ldots\sigma_{1_{l}}^{\varepsilon_{l}}$ in Artin generators $\sigma_{i}$ such that for all $i$ there is a $j\in\{1,2,\ldots,l\}$ with $i_{j}=i$ and if $i_{j}=i_{k}$, then $\varepsilon_{j}=\varepsilon_{k}$.
\end{definition}  

In other words, a strictly homogeneous braid has a word presentation where each generator $\sigma_{i}$ appears in the braid word if and only if $\sigma_{i}^{-1}$ does not. 
In order to show Theorem \ref{even} we make use of a lemma shown in \cite{bode:2016polynomial}.

\begin{lemma}
\label{crit}
(cf. \cite{bode:2016polynomial})
Let $w$ be a strictly homogeneous braid. Then there exists a finite Fourier parametrisation of $w$ as in Equation (\ref{eq:fourier}) such that the resulting $g_{a,b}$ does not have any argument-critical points, i.e.\ for all $(u,t)\in\mathbb{C}\times[0,2\pi]$ the Jacobian $\nabla_{\mathbb{R}^{3}}arg(g_{a,b,})(u,t)\neq (0,0,0)$.
\end{lemma}

Together with the opening remarks to this section Lemma \ref{crit} implies a result by Stallings \cite{stallings:1978constructions} that every closure of a strictly homogeneous braid is fibered.

In practice finding a braid parametrisation as in Lemma \ref{crit} is challenging. We denote the $s-1$ critical points of $g_{a,b}(\bullet, t)$ by $u_{j,t}$ and the critical values $g_{a,b}(u_{j,t},t)$ by $v_{j,t}$. Finding a parametrisation as in Lemma \ref{crit} involves parametrising the critical values $v_{j}(t)$, $t\in[0,2\pi]$, $j=1,2,\ldots,s-1$ such that $\partial\arg(v_{j}(t))/\partial t\neq 0$ subject to topological constraints on the braid that is formed by the different $v_{j}(t)$ and a strand whose complex coordinate is zero for all $t\in[0,2\pi]$. More details on this can be found in \cite{bode:2016polynomial} and \cite{rudolph:2001some}. Then for each $t\in[0,2\pi]$ we need to solve $s-1$ polynomial equations in $s-1$ complex variables in order to find families of complex polynomials $p_{t}(u)$ which have $v_{j}(t)$ as their critical values for all $t\in[0,2\pi]$. Fixing one degree of freedom, there are $s^{s-1}$ solutions to these equations for each $t$ \cite{bcn:2002critical}, leading to $s^{s-1}$ families of polynomials $p_{t}(u)$ which have $v_{j}(t)$ as their critical values. It turns out that the roots of one of these $p_{t}(u)$ will form the desired braid, which can then be approximated by trigonometric polynomials.

Suppose now that $B$ is the square of another braid word, say $B=w^{2}$. Let 
\begin{equation}
\label{eq:fourier2}\bigcup_{C\in\mathcal{C}}\bigcup_{j=1}^{s_{C}}\left(F_{C}\left(\frac{t+2\pi j}{s_{C}}\right),G_{C}\left(\frac{t+2\pi j}{s_{C}}\right),t\right)
\end{equation}
be a finite Fourier parametrisation of the braid $w$. Then 
\begin{equation}
\label{eq:evenpara}
\bigcup_{C\in\mathcal{C}}\bigcup_{j=1}^{s_{C}}\left(F_{C}\left(\frac{2t+2\pi j}{s_{C}}\right),G_{C}\left(\frac{2t+2\pi j}{s_{C}}\right),t\right)
\end{equation}
is a finite Fourier parametrisation of the braid $B$. Furthermore, if we use the parametrisation given in Equation (\ref{eq:evenpara}) to define $g_{a,b}$ and $p_{a,b,k}$ as in Lemma \ref{weakest} we find that all exponents of $\sqrt{v\overline{v}}$ in $p_{a,b,k}$ are even. Thus working with this parametrisation ensures that $p_{a,b,k}$ is a polynomial in $u$, $v$ and $\overline{v}$ as shown in Lemma \ref{evenweak}.

By Lemma \ref{weakest} the function $p_{a,b,k}$ has a weakly isolated singularity at the origin. 

Note that if $w$ has a parametrisation as in Equation (\ref{eq:fourier2}) such that the corresponding $g_{a,b}$ does not have any argument-critical points, then Equation (\ref{eq:evenpara}) is a parametrisation of $B=w^2$ such that the function $g_{a,b}$ corresponding to this parametrisation does not have any argument-critical points either. Thus all squares $B=w^2$ of strictly homogeneous braids $w$ have a parametrisation as in Equation (\ref{eq:evenpara}) such that the corresponding $g_{a,b}$ does not have any argument-critical points and the corresponding $p_{a,b,k}$ with $q_{1}=q_{2}=0$ is a polynomial in $u$, $v$ and $\overline{v}$ with an isolated singular point at the origin.

We now use Lemma \ref{crit} to show the following lemma.
\begin{lemma}
\label{isolated}
Let $B=w^2$, where $w$ is a strictly homogeneous braid. Then there exists a finite Fourier parametrisation of $B$ as in Equation (\ref{eq:fourier}) such that the resulting $p_{a,b,k}$ with $q_{1}=q_{2}=0$ has an isolated singularity at the origin.
\end{lemma}

\begin{proof}
As discussed above, $B$ has a finite Fourier parametrisation such that the corresponding $g_{a,b}$ does not have any argument-critical points and the corresponding $p_{a,b,k}$ is a polynomial in $u$, $v$ and $\overline{v}$.

Assume $(u,v)$ is a critical point of $p_{a,b,k}$ that is not the origin. Then since $p_{a,b,k}$ is a polynomial in $u$, it must be $\frac{\partial p_{a,b,k}}{\partial u}(u,v)=0$. Otherwise the Cauchy-Riemann equations would guarantee that the Jacobian of $p_{a,b,k}$ at $(u,v)$ has full rank. Since for $v=0$ we have that $p_{a,b,k}(u,0)=u^{s}$, a critical point which is not the origin must have non-zero $v$.

It follows from $p_{a,b,k}(u,re^{it})=r^{2sk}g_{a,b}(u/r^{2k},t)$ that for every $u\in\mathbb{C}$, $v=re^{it}\in\mathbb{C}\backslash\{0\}$ with $\frac{\partial p_{a,b,k}}{\partial u}(u,v)=0$ we have 
\begin{equation}\frac{\partial p_{a,b,k}}{\partial arg(v)}(u,re^{it})=r^{2sk}\frac{\partial g_{a,b}}{\partial t}(u,e^{it})\neq 0,
\end{equation}
since $g_{a,b}$ does not have any argument-critical points.

Consider now the Jacobian of $p_{a,b,k}$ at a point $(u,v)\in\mathbb{C}\times(\mathbb{C}\backslash \{0\})$ with $\frac{\partial p_{a,b,k}}{\partial u}(u,v)=0$ using $(\mathrm{Re}(u),\mathrm{Im}(u),|v|,\arg(v))$ as a basis for $\mathbb{R}^4$ (whether the Jacobian has full rank or not does not depend on the choice of basis). We already know that in this basis the Jacobian has the form  
\begin{equation}
\label{eq:jacobian}
\nabla p_{a,b,k}(u,v)=
\begin{pmatrix}
0 & 0 & \alpha & \beta\\
0 & 0 & \gamma & \delta
\end{pmatrix},
\end{equation}
where $\alpha=\partial \mathrm{Re}(p_{a,b,k})/\partial r$, $\beta=\partial \mathrm{Re}(p_{a,b,k})/\partial \arg(v)$, $\gamma=\partial \mathrm{Im}(p_{a,b,k})/\partial r$,\\ $\delta=\partial \mathrm{Im}(p_{a,b,k})/\partial \arg(v)$ and 
\begin{equation}
\label{eq:notzero}
\mathrm{Im}(p_{a,b,k})\beta-\mathrm{Re}(p_{a,b,k})\delta\neq 0.
\end{equation}

We calculate 
\begin{align}
\frac{\partial p_{a,b,k}}{\partial r}(u,re^{it})&=r^{2sk}\frac{\partial f_{a,b}}{\partial r}\left(\frac{u}{r^{2k}},e^{it}\right)+2skr^{2sk-1}f_{a,b}\left(\frac{u}{r^{2k}},e^{it}\right)\nonumber \\
&=-2kur^{2k(s-1)-1}\frac{\partial f_{a,b}}{\partial u}+2skr^{2sk-1}f_{a,b}\left(\frac{u}{r^{2k}},e^{it}\right)\nonumber \\
&=2skr^{-1}p_{a,b,k}(u,re^{it}),
\end{align}
where the last equality follows from $\frac{\partial f_{a,b}}{\partial u}\left(\frac{u}{r^{2k}},e^{it}\right)=r^{2k(1-s)}\frac{\partial p_{a,b,k}}{\partial u}(u,re^{it})=0$.

But this means that by Equation (\ref{eq:notzero})
\begin{align}
\beta\gamma-\alpha\delta&=2skr^{2sk-1}\frac{\partial \mathrm{Im}(p_{a,b,k}}{\partial r}(u,vre^{it}))\beta-2skr^{2sk-1}\frac{\partial \mathrm{Re}(p_{a,b,k}}{\partial r}(u,v))\delta\nonumber\\
&=2skr^{-1}(\mathrm{Im}(p_{a,b,k}(u,v))\beta-\mathrm{Re}(p_{a,b,k}(u,v))\delta)\neq 0.
\end{align}
Hence $\left(\begin{smallmatrix}\alpha & \beta\\ \gamma & \delta\end{smallmatrix}\right)$ has rank 2 and $(u,v)$ is a regular point. Thus the origin is the only critical point of $p_{a,b,k}$ and therefore isolated.
\end{proof}

Lemma \ref{isolated} concludes the proof Theorem \ref{even}. We have shown that closures of even powers of strictly homogeneous braids are real algebraic.

\subsection{More constructions of real algebraic links}

In Section \ref{weak} we introduced two classes of braid parametrisations that lead to $p_{a,b,k}$ being a polynomial map if $q_{1}$ and $q_{2}$ are chosen appropriately. The first of these classes is the set of squares of braids (cf. Lemma \ref{evenweak}) and the second class consists of the braids whose parametrisations satisfy certain arithmetic conditions (cf. Lemma \ref{odd}). For the first class we showed in the proof of Lemma \ref{isolated} that one extra condition, namely the absence of argument-critical points, is sufficient to guarantee that $p_{a,b,k}$ has an isolated singular point. In this section we investigate the effects of such a parametrisation of braids that belong to the second class.

In the proof of Lemma \ref{isolated} the stretching parameters $a$ and $b$ are constants. In the following they depend on $r=|v|$, since for the relevant polynomial $p_{a,b,k}$ for braid parametrisations of the second class $q_{1}$ and $q_{2}$ are not both zero. While large parts of the proof of Lemma \ref{isolated} remain unchanged, this fact implies that the derivatives with respect to $r$ differ.  



Recall that the arithmetic condition in Lemma \ref{odd} requires a finite Fourier parametrisation of the braid where for both the $x$-coordinate and the $y$-coordinate all frequencies with non-zero coefficients are in the same residue class $\ \mathrm{mod}\  2^{m+1}$, where $2^m$ is the largest power of 2 dividing $s_{C}$, the number of strands in each link component. 

This way every parametrisation of this form specifies two residue classes $2^{m+1}$, one for the $x$-coordinate and for the $y$-coordinate. While in general these two classes are not identical, we begin by studying the special case where they are.

\begin{lemma}
\label{same}
Let $B$ be a braid with a parametrisation as in Lemma \ref{odd} such that all $j\in\{0,1,\ldots,\max\{N_{C},M_{C}\}\}$ with non-vanishing $a_{C,j}$ or $b_{C,j}$ are in the same residue class $\ \mathrm{mod}\  2^{m+1}$, i.e. $x=y$, and such that the function $g_{a,b}$ corresponding to this parametrisation does not have any argument-critical points for some $a,b>0$. Then $p_{a,b,k}$ with $q_{1}=q_{2}=x/2^m$ has an isolated singular point at the origin and the vanishing set of $p_{a,b,k}$ on $S^{3}_{\rho}$ is the closure of $B$ for all positive $\rho\leq 1$.
\end{lemma}

\begin{proof}
By Lemma \ref{odd} $p_{a,b,k}$ is a polynomial in $u$, $v$ and $\overline{v}$. Note that if $x=y$, then
\begin{equation}
p_{a,b,k}(u,v)=(v\overline{v})^{s(k+x/2^{m})}f_{a,b}\left(\frac{u}{(v\overline{v})^{k+x/2^{m}}},\frac{v}{\sqrt{v\overline{v}}}\right)
\end{equation} 

and hence
\begin{equation}
\label{eq:unit}
p_{a,b,k}(u,re^{it})=g_{r^{2(k+x/2^{m})}a,r^{2(k+x/2^{m})}b}(u,t)=r^{2s(k+x/2^{m})}g_{a,b}\left(\frac{u}{r^{2(k+x/2^{m})}},e^{it}\right)
\end{equation}
and $p_{a,b,k}(u,0)=u^{s}$.

It is easy to see that the origin is a singular point of $p_{a,b,k}$. 
As shown in Lemma \ref{general} the vanishing set of $p_{a,b,k}$ on $S^{3}_{\rho}$ is the closure of $B$ for all positive $\rho\leq1$.

What is left to show is that the singular point at the origin is isolated. This is similar to the proof of Lemma \ref{isolated}. Again every critical point of $p_{a,b,k}$ that is not the origin must satisfy $v\neq0$ and $\partial p_{a,b,k}/\partial u=0$. Since $g_{a,b}$ does not have any argument-critical points, $g_{\lambda a,\lambda b}$ does not have any argument-critical points either for any $\lambda>0$. Hence the Jacobian $\nabla p_{a,b,k}$ is of the form of Equation (\ref{eq:jacobian}) with $\mathrm{Im}(p_{a,b,k})\beta-\mathrm{Re}(p_{a,b,k})\delta\neq0$.

The argument that the singular point at the origin is isolated is now seen to be exactly the same as in the proof of Lemma \ref{isolated}.

\end{proof}

In Lemma \ref{same} the values of the scaling parameters $a$ and $b$ are not constant anymore, but depend on $r=|v|$. This dependence however, is exactly the same for $a$ and $b$, namely $r^{x/2^{m}}$ times a constant.
If $x\neq y$, the derivatives with respect to $r=|v|$ now change in a non-trivial way compared to the derivatives in the proof of Lemma \ref{same}. Therefore, the described method of constructing isolated singularities does not necessarily work for these links.  

Note that as in Theorem \ref{even} the polynomials maps constructed in Lemma \ref{same} are semiholomorphic in $u$ and their degree with respect to $u$ is equal to the number of strands $s$ used in the construction. In contrast to the construction of semiholomorphic polynomials with knotted vanishing sets on the unit three-sphere we can not give a bound on the degree with respect to $v$ and $\overline{v}$. This is due to the fact that because of the several conditions that the braid parametrisations have to satisfy we can not use trigonometric interpolation to find a parametrisation of the desired form.  

It is not clear how Lemma \ref{same} could be used to decide whether a given link $L$ is real algebraic. At least we are not aware of any algorithm that determines if $L$ is the closure of a braid of the desired form.
 
However, Lemma \ref{same} offers a way of constructing real algebraic links that are not necessarily of the form of Lemma \ref{isolated}. For each Fourier parametrisation that satisfies the arithmetic condition, we can try to find values of $a, b>0$ such that $\arg g_{a,b}$ does not have any critical points. This might not always be possible, but if we can do this, then the closure of the braid parametrised in this way is real algebraic. The next section is devoted to following this procedure for particularly simple Fourier parametrisations.

\section{Real algebraic lemniscate knots}
\label{lemniscatesection}

We discussed lemniscate knots in \cite{bode:2016lemniscate} as a family of links that have a parametrisation as in Equation (\ref{eq:para}) of a  particularly simple form. The $(s,\ell,r)$-lemniscate link is the closure of the braid given by
\begin{equation}
\label{eq:lemnipara}
\bigcup_{j=0}^{s-1}\left(\cos\left(\frac{rt+2\pi j}{s}\right),\sin\left(\frac{\ell(rt+2\pi j)}{s}\right),t\right)\qquad t\in[0,2\pi].
\end{equation} 

Several properties of this family of links can be found in \cite{bode:2016lemniscate}, among others that a braid parametrised by Equation (\ref{eq:lemnipara}) has a braid word of the form $(\sigma_{1}^{\varepsilon_{1}}\sigma_{3}^{\varepsilon_{3}}\ldots\sigma_{2}^{\varepsilon_{2}}\sigma_{4}^{\varepsilon_{4}}\ldots)^{r}$, where $\varepsilon_{j}\in\{\pm1\}$ depends on $\ell$ and $s$. In particular, this braid word is strictly homogeneous and hence there exists a Fourier parametrisation of this braid whose corresponding function $g_{a,b}$ does not have any phase-critical points for some $a,b>0$. However, this parametrisation might be different from the one given in Equation (\ref{eq:lemnipara}).

Lemma \ref{general} and Lemma \ref{odd} give us a way of explicitly constructing any lemniscate link as the link of a weakly isolated singularity, since the defining parametrisation obviously satisfies the conditions of Lemma \ref{odd}.

The lemniscate links for which Lemma \ref{same} implies the stronger notion of isolation of the singular point must satisfy 
\begin{equation}
\label{eq:relation}
r\ell\equiv r\ \mathrm{mod}\  2^{m+1},
\end{equation} 
where $2^m$ is the largest power of 2 dividing $s_{C}=s/\gcd(s,r)$. Since all lemniscate links are closures of strictly homogeneous braids, all lemniscate links with even $r$ are real algebraic and can be constructed as links of isolated real singularities as described in Section \ref{real}. We can hence assume that $r$ is odd and thus $2^m$ is the largest power dividing $s$.

If $r$ is odd, then it is coprime to $2^{m+1}$ and hence Equation (\ref{eq:relation}) implies $\ell\equiv 1\ \mathrm{mod}\  2^{m+1}$. These are all lemniscate links that Lemma \ref{same} can potentially be applied to and that are not already covered by Section \ref{real}.

We encounter one typical intricacy in the application of Lemma \ref{same}. As we have seen, a lemniscate link with $\ell\equiv 1\ \mathrm{mod}\  2^{m+1}$ has a Fourier parametrisation that satisfies the desired arithmetic properties. That same braid also has a parametrisation that satifies the condition that $g_{a,b}$ does not have any argument-critical points. In order to use Lemma \ref{same} we need to find a parametrisation that satisfies both conditions simultaneously.

Note that by Lemma \ref{same} finding a value for $b$ such that $g_{1,b}$, constructed as in Equation (\ref{eq:gab}) from Equation (\ref{eq:lemnipara}), does not have any argument-critical points is sufficient to show that the closure of the braid is real algebraic.
Finding such a $b$ for the $(s,\ell,r)$-lemniscate link for $r=1$ shows that every $(s,\ell,r)$-lemniscate link is real algebraic.

We denote by $u_{k}(t)$, $k=1,2,\ldots,s-1$ the $s-1$ solutions of $\frac{\partial g_{1,b}}{\partial t}(u,t)=0$ for a given $t\in[0,2\pi]$ and need to check that 
\begin{equation}\frac{\partial \arg g_{1,b}(u_{k}(t),t)}{\partial t}=\mathrm{Im} \sum_{j=1}^{s}\frac{\frac{1}{s}\sin\left(\frac{t+2\pi j}{s}\right)-\frac{ib\ell}{s}\cos\left(\frac{\ell(t+2\pi j)}{s}\right)}{u_{k}(t)-\cos\left(\frac{t+2\pi j}{s}\right)-ib\sin\left(\frac{\ell(t+2\pi j)}{s}\right)}\neq 0
\end{equation}
for all $k=1,2,\ldots,s-1$ and all $t\in[0,2\pi]$. 

In the case of $\ell=3$ and $s=5$, which are the lowest numbers satisfying $\ell\equiv 1\ \mathrm{mod}\  2^{m+1}$ that do not lead to torus links, we find numerically that it is sufficient to let $b$ equal $1/4$. Thus the $(s=5,\ell=3,r)$-lemniscate link is real algebraic for every $r$.

This includes several examples of real algebraic links that are not covered by Theorem \ref{even}.

Preliminary numerical investigation indicates that for other lemniscate links small enough choices of $b$ again lead to braid polynomials $g_{1,b}$ without any argument-critical points. We may thus conjecture that all lemniscate links with $\ell\equiv 1\ \mathrm{mod}\  2^{m+1}$ are real algebraic, but at the moment this remains an open problem.

An analytic proof of this would be desirable, but note that for $b=0$ the function $\frac{\partial \arg g_{1,0}}{\partial t}(u_{k}(t),t)$ is piecewise constant where it is defined. This means that using limit arguments becomes challenging.

It becomes increasingly harder to determine sufficient values of $b$ numerically as $\ell$ and $s$ increase, since it involves finding the roots of a continuous family of polynomials of degree $s-1$.

\section{The strong Milnor condition}
\label{milnor}
For a complex plane curve $f:\mathbb{C}^2\to\mathbb{C}$ with an isolated singularity at the origin $f/|f|$ automatically is a fibration of $S^3_{\epsilon}\backslash f^{-1}(0)$ over $S^1$ for small enough $\epsilon>0$. Even though real algebraic links are fibered, it is not always the case that $f/|f|$ is a fibration of $S^3_{\epsilon}\backslash f^{-1}(0)$ when $f:\mathbb{R}^4\to\mathbb{R}^2$ is a polynomial with an isolated singularity.

The following definitions and Theorem \ref{dimp} can be found in \cite{cisneros:2010real} and \cite{cisneros:2011real}.


\begin{definition}
Let $f:\mathbb{R}^4\to\mathbb{R}^2$ be a polynomial map with isolated singular point at the origin. Then $f$ satisfies the \textit{strong Milnor condition} if there exists an $\epsilon>0$ such that 
\begin{equation}
f/|f|:S^{3}_{\rho}\backslash f^{-1}(0)\to S^1
\end{equation}
is a fibration for all $0<\rho<\epsilon$.
\end{definition}

In this section we show that the maps $p_{a,b,k}$ as in Lemma \ref{isolated} and Lemma \ref{same} satisfy the strong Milnor condition. We start with a definition.
\begin{definition}
Let $f:\mathbb{R}^4\to\mathbb{R}^2$ be a polynomial map with an isolated singularity at the origin and $\mathcal{L}_{\ell}\subset \mathbb{R}^2$ be the line through the origin corresponding to $\ell\in\mathbb{RP}^1$. Then we define $X_{\ell}=\{x\in\mathbb{R}^4:f(x)\in \mathcal{L}_{\ell}\}$ and call $X=\{X_{\ell}:\ell\in\mathbb{RP}^1\}$ the canonical pencil.
\end{definition}

Note that with this definition each $X_{\ell}$ is a 3-dimensional manifold and the different $X_{\ell}$ meet at $f^{-1}(0)$.

\begin{definition}
We say a map $f:\mathbb{R}^4\to\mathbb{R}^2$ is $d$-regular with respect to the standard metric on $\mathbb{R}^4$ if there exists an $\epsilon>0$ such that every three-sphere $S^{3}_{\rho}$ of radius $\rho<\epsilon$ intersects all $X_{\ell}\backslash f^{-1}(0)$ transversely (provided the intersection is non-empty).
\end{definition}

The property of $d$-regularity can be used to study if a given map satisfies the strong Milnor condition.
\begin{theorem}
\label{dimp}
(cf. \cite{cisneros:2010real})
If a polynomial map $f:\mathbb{R}^4\to\mathbb{R}^2$ with an isolated singularity at the origin is $d$-regular with respect to the standard metric, then it satisfies the strong Milnor condition.
\end{theorem}

The concept of $d$-regularity can also be defined for metrics that are not the standard metric. In fact, $d$-regularity for some metric induced by a positive definite quadratic form is equivalent to the strong Milnor condition. In the case of the maps that we constructed in Section \ref{real} it is sufficient to consider the standard metric.

\begin{proposition}
\label{trans}
Let $p_{a,b,k}$ as constructed in Lemma \ref{isolated} or Lemma \ref{same}. Then $p_{a,b,k}$ is $d$-regular with respect to the standard metric.
\end{proposition}

\begin{proof}
We need to show that for all small enough radii $\rho$ and all $\ell\in\mathbb{RP}^1$ the intersection $S^{3}_{\rho}\cap X_{\ell}$ is either empty or transverse.

Let $(u,v)\in S^{3}_{\rho}\cap X_{\ell}$. There are three different cases to consider: one where $u=0$, one where $v=0$ and the remaining case where both are non-zero.

If $u=0$, then $p_{a,b,k}(0,re^{it})$ is for every fixed $t\in[0,2\pi]$ by definition either constant zero (if $g_{a,b}(0,t)=0$) or the argument $\arg(p_{a,b,k})(0,re^{it})$ is constant for all $r>0$. Hence in the basis $(\mathrm{Re}(u),\mathrm{Im}(u),r,t)$ the vector $(0,0,1,0)$ is tangent to $X_{\ell}$ at $(0,v)$. Thus the intersection with $S^{3}_{\rho}$ is transverse.

If $v=0$, then $p_{a,b,k}(u,0)=u^s$. Hence in the basis $(|u|,\arg(u),\mathrm{Re}(v),\mathrm{Im}(v))$ the vector $(1,0,0,0)$ is a tangent to $X_{\ell}$ at $(u,0)$ and hence the intersection with $S^{3}_{\rho}$ is transverse.

Now suppose both $u$ and $v$ are non-zero. Then we can use $(|u|,\arg(u),r,t)$ as a basis and by definition of $p_{a,b,k}$ the vector $(r^{2k+x/2^m},0,1,0)$ is tangent to $X_{\ell}$ at $(u,v)$. The tangent space of $S^{3}_{\rho}$ at $(u,v)$ is spanned by $(0,1,0,0)$, $(0,0,0,1)$ and $(-r/\sqrt{\rho^2-r^2},0,1,0)$. Thus the intersection is not transverse if and only if $r^{2k+x/2^m}=-r/\sqrt{\rho^2-r^2}$. Since $r>0$, this equality is never satisfied, which completes the proof of $d$-regularity of $p_{a,b,k}$.
\end{proof}

The next corollary follows from Theorem \ref{dimp} and Proposition \ref{trans}.

\begin{corollary}
The maps constructed in Lemma \ref{isolated} and Lemma \ref{same} satisfy the strong Milnor condition.
\end{corollary}

We have thus shown that for small enough radii $\rho$
\begin{equation}
p_{a,b,k}/|p_{a,b,k}|:S^{3}_{\rho}\backslash p_{a,b,k}^{-1}(0)\to S^1
\end{equation}
is a fibration. 

We hope that the explicit construction of polynomials and fibrations for the links that were shown to be real algebraic in Section \ref{real} helps to investigate properties of real algebraic links and of their fibrations.

\begin{acknowledgements}
The author is grateful to Gareth Alexander, Mark Bell, Mark Dennis, David Foster, Mikami Hirasawa, Filip Misev, Daniel Peralta-Salas, Jonathan Robbins and De Witt Sumners for discussions and comments.

This work is funded by the Leverhulme Trust Research Programme Grant RP2013- K-009, SPOCK: Scientific Properties Of Complex Knots.
\end{acknowledgements}



\end{document}